\documentclass[12pt]{amsart}
\usepackage[dvips]{graphicx}
\usepackage{amsmath}
\usepackage{amssymb}
\usepackage{amscd}
\usepackage{mathrsfs}
\usepackage{mathtools}
\usepackage{color}
\usepackage[all]{xy}%
\DeclareFontEncoding{OT2}{}{}
\DeclareFontSubstitution{OT2}{cmr}{m}{n}
\DeclareFontFamily{OT2}{cmr}{\hyphenchar\font45 }
\DeclareFontShape{OT2}{cmr}{m}{n}{%
   <5><6><7><8><9>gen*wncyr%
   <10><10.95><12><14.4><17.28><20.74><24.88>wncyr10}{}
\DeclareFontShape{OT2}{cmr}{b}{n}{%
   <5><6><7><8><9>gen*wncyb%
   <10><10.95><12><14.4><17.28><20.74><24.88>wncyb10}{}
\DeclareMathAlphabet{\mathcyr}{OT2}{cmr}{m}{n}
\DeclareMathAlphabet{\mathcyb}{OT2}{cmr}{b}{n}
\SetMathAlphabet{\mathcyr}{bold}{OT2}{cmr}{b}{n}

\theoremstyle{plain}
    \newtheorem{theorem}{Theorem}[section]
\theoremstyle{definition}
    \newtheorem{definition}[theorem]{Definition}
    \newtheorem{proposition}[theorem]{Proposition}
    \newtheorem{lemma}[theorem]{Lemma}
    \newtheorem{corollary}[theorem]{Corollary}

    \newtheorem{remark}[theorem]{Remark}
    
\newcommand{\bZ}{\mathbb{Z}}
\newcommand{\bQ}{\mathbb{Q}}
\newcommand{\bR}{\mathbb{R}}

\newcommand{\cA}{\mathcal{A}}
\newcommand{\cB}{\mathcal{B}}
\newcommand{\cF}{\mathcal{F}}
\newcommand{\cS}{\mathcal{S}}
\newcommand{\cZ}{\mathcal{Z}}
\newcommand{\frH}{\mathfrak{H}}

\newcommand{\frS}{\mathfrak{S}}
\newcommand{\frZ}{\mathfrak{Z}}
\newcommand{\sh}{\mathbin{\mathcyr{sh}}}

\newcommand{\bk}{\boldsymbol{k}}
\newcommand{\bl}{\boldsymbol{l}}

\newcommand{\pp}{\boldsymbol{p}}
\newcommand{\reg}{\mathrm{reg}}
\newcommand{\dep}{\mathrm{dep}}
\newcommand{\wt}{\mathrm{wt}}
\makeatletter
    
    \@addtoreset{equation}{section}
  \makeatother
\newcommand{\jump}[1]{\ensuremath{[\![#1]\!]} }

\address[Masataka Ono]{Global Education Center, Waseda University, 1-6-1 Nishi-Waseda, Shinjuku-ku, Tokyo, 169-8050, Japan}
\email{m-ono@aoni.waseda.jp}

\address[Kosuke Sakurada]{Seiwa Gakuen High School, 3-4-1, Kinoshita, Wakabayashi-ku, Sendai, 984-0047, Japan}
\email{sakurada.kosuke@seiwa.ac.jp}

\address[Shin-ichiro Seki]{Department of Mathematical Sciences\\ Aoyama Gakuin University\\ 5-10-1 Fuchinobe, Chuo-ku, Sagamihara-shi, Kanagawa, 252-5258, Japan}
\email{seki@math.aoyama.ac.jp}
\title{A note on $\cF_n$-multiple zeta values}
\author{Masataka Ono, Kosuke Sakurada and Shin-ichiro Seki}
\thanks{This research was supported in part by JSPS KAKENHI Grant Numbers JP16H06336 and 18J00151.}
\begin{document}
\maketitle
\begin{abstract}
For several evaluations of special values and several relations known only in $\cA_n$-multiple zeta values or $\cS_n$-multiple zeta values, we prove that they are uniformly valid in $\cF_n$-multiple zeta values for both the case where $\cF=\cA$ and $\cF=\cS$.
In particular, the Bowman--Bradley type theorem and sum formulas for $\cS_2$-multiple zeta values are proved.
\end{abstract}
\section{Introduction}
We call a tuple of positive integers $\bk=(k_1,\dots,k_r)$ an \emph{index}. We call $\wt(\bk)\coloneqq k_1+\cdots+k_r$ (resp.~$\dep(\bk)\coloneqq r$) the \emph{weight} (resp.~\emph{depth}) of $\bk$.
If the condition $k_r\geq 2$ is satisfied, then we state that the index $\bk=(k_1,\dots, k_r)$ is \emph{admissible}.
For an admissible index $\bk=(k_1,\dots, k_r)$, the \emph{multiple zeta value} (MZV) $\zeta(\bk)$ and the \emph{multiple zeta-star value} (MZSV) $\zeta^{\star}(\bk)$ are defined by 
\begin{align*}
\zeta(\bk)
\coloneqq \sum_{0<n_1<\cdots<n_r}\frac{1}{n^{k_1}_1\cdots n^{k_r}_r}, \qquad 
\zeta^\star(\bk)\coloneqq \sum_{1\leq n_1\leq \cdots \leq n_r}\frac{1}{n^{k_1}_1\cdots n^{k_r}_r}.
\end{align*}
These series are convergent.
We set $\zeta(\varnothing)=\zeta^\star(\varnothing)=1$ for the empty index $\varnothing$ (= the empty tuple). 

First we recall the definition of \emph{$\cA_n$-multiple zeta(-star) values} ($\cA_n$-MZ(S)Vs) introduced by Rosen; see \cite{Ro1,Se2}.
For a positive integer $n$, set
\begin{align*}
\cA_n\coloneqq \left.\prod_p\bZ/p^n\bZ\right/\bigoplus_p\bZ/p^n\bZ,
\end{align*}
where $p$ runs over all prime numbers.
For an index $\bk=(k_1, \ldots, k_r)$, the $\cA_n$-MZV $\zeta_{\cA_n}^{}(\bk)$ and the $\cA_n$-MZSV $\zeta_{\cA_n}^{\star}(\bk)$ are defined by
\begin{align*}
\zeta^{}_{\cA_n}(\bk)\coloneqq \left(\sum_{0<n_1<\cdots<n_r<p}\frac{1}{n^{k_1}_1\cdots n^{k_r}_r} \bmod{p^n}\right)_p, 
\end{align*}
\begin{align*}
\zeta^\star_{\cA_n}(\bk)\coloneqq \left(\sum_{1\leq n_1\leq \cdots\leq n_r\leq p-1}\frac{1}{n^{k_1}_1\cdots n^{k_r}_r} \bmod{p^n}\right)_p 
\end{align*}
as elements of $\cA_n$.
We also set $\zeta^{}_{\cA_n}(\varnothing)=\zeta^\star_{\cA_n}(\varnothing)=1$.

Next we recall the definition of \emph{$t$-adic symmetric multiple zeta values} ($\widehat{\cS}$-MZVs) introduced by Jarossary \cite{J2}.
Let $t$ be an indeterminate.
For $\bullet \in \{*, \sh\}$ and an index $\bk=(k_1,\dots, k_r)$, set
\begin{align*}
\zeta^\bullet_{\widehat{\cS}}(\bk)&=\sum_{i=0}^r(-1)^{k_{i+1}+\cdots+k_r}\zeta^\bullet(k_1, \ldots, k_i)\\
&\quad\times\sum_{l_{i+1}, \ldots, l_r\geq0}\left[\prod_{j=i+1}^r\binom{k_j+l_j-1}{l_j}\right]\zeta^\bullet(k_r+l_r, \ldots, k_{i+1}+l_{i+1})t^{l_{i+1}+\cdots+l_r} \in \cZ\jump{t}.
\end{align*}
Here, $\cZ$ is the $\bQ$-subalgebra of $\bR$ generated by all MZVs and $\zeta^*(\bk) \in \cZ$ (resp.~$\zeta^{\sh}(\bk) \in \cZ$) is the harmonic (resp.~shuffle) regularized MZV.
See Subsection~\ref{subsec:alg_setup} for details.
It is known that $\zeta^*_{\widehat{\cS}}(\bk)-\zeta^{\sh}_{\widehat{\cS}}(\bk) \in (\zeta(2)\cZ)\jump{t}$ for any index $\bk$ (\cite[Proposition~3.2.4]{J2} and \cite[Proposition~2.1]{OSY}).
Thus, 
\begin{align*}
\zeta^{}_{\widehat{\cS}}(\bk)\coloneqq \zeta^\bullet_{\widehat{\cS}}(\bk) \bmod{\zeta(2)}
\end{align*}
is independent of the choice of the regularization $\bullet \in \{*, \sh\}$ and defines a well-defined element of $\overline{\cZ}\jump{t}\coloneqq (\cZ/\zeta(2)\cZ)\jump{t}$.
We call $\zeta^{}_{\widehat{\cS}}(\bk)$ the $\widehat{\cS}$-MZV.
We also define the \emph{$t$-adic symmetric multiple zeta-star value} ($\widehat{\cS}$-MZSV) $\zeta^\star_{\widehat{\cS}}(\bk)$ by
\[
\zeta^{\star}_{\widehat{\cS}}(k_1,\dots, k_r)=\sum_{\substack{\square \  \text{is either a comma `,'} \\ \text{or a plus `$+$'}}}\zeta^{}_{\widehat{\cS}}(k_1\square \cdots \square k_r).
\]
See \cite[Definition~1.1]{HMO} for another equivalent definition of the $\widehat{\cS}$-MZSV.
For a positive integer $n$, let $\pi_n\colon\overline{\cZ}\jump{t}\twoheadrightarrow\overline{\cZ}\jump{t}/(t^n)$ be the natural projection.
\begin{definition}
For an index $\bk=(k_1, \ldots, k_r)$, we define the \emph{$\cS_n$-multiple zeta(-star) value} ($\cS_n$-MZ(S)V) by 
\begin{align*}
\zeta^{}_{\cS_n}(\bk)\coloneqq \pi_n\bigl(\zeta^{}_{\widehat{\cS}}(\bk)\bigr), \quad \zeta^\star_{\cS_n}(\bk)\coloneqq \pi_n\bigl(\zeta^\star_{\widehat{\cS}}(\bk)\bigr)=\sum_{\substack{\square \  \text{is either a comma `,'} \\ \text{or a plus `$+$'}}}\zeta^{}_{\cS_n}(k_1\square \cdots \square k_r).
\end{align*}
\end{definition}
Note that $\zeta^{}_{\cS_1}(\bk)$ coincides with the usual symmetric multiple zeta value (SMZV) $\zeta^{}_{\cS}(\bk)$ defined by Kaneko and Zagier \cite{KZ}.

$\cA_n$-MZ(S)Vs and $\cS_n$-MZ(S)Vs are the main objects of this article and together they are called $\cF_n$-MZ(S)Vs; $\cF$ derives from the first letter of the word ``finite''.
Similar to the conjecture \cite[Conjecture~4.3]{OSY}, it is conjectured that $\cA_n$-MZVs and $\cS_n$-MZVs satisfy relations of the same form.
Hence, a relation among $\cA_n$-MZVs or $\cS_n$-MZVs is always described collectively as a relation of $\cF_n$-MZVs, at least conjecturally.
The purpose of this paper is to confirm that several evaluations of special values and several relations known only in $\cA_n$-MZVs or $\cS_n$-MZVs are uniformly valid in $\cF_n$-MZVs.
In some cases, we only deal with $n=1,2,3$.

The remainder of the paper is structured as follows.
In Section~\ref{sec:Prelim}, we prepare relevant tools including Zagier's formula for MZVs, the double shuffle relation for $\cF_n$-MZVs and the relation for $\cF_n$-MZVs derived from the antipode.
In Section~\ref{sec:sp_value}, we put forward some explicit evaluations of $\cF_n$-MZ(S)Vs.
In Section~\ref{sec:BB}, we prove the Bowman--Bradley type theorem for $\cF_2$-MZ(S)Vs.
In Section~\ref{sec:sum_formulas}, we prove sum formulas for $\cF_n$-MZ(S)Vs with respect to specific $n$.
Some complicated but elementary calculations for binomial coefficients (= the proof of Proposition~\ref{prop:true value of C}) are proved in the Appendix.
\section*{acknowledgement}
The authors would like to thank Dr.~Kenji Sakugawa for helpful discussion concerning Proposition~\ref{prop:Z_[A_n]} and Theorem~\ref{thm:dep-1_general}.
The authors also would like to thank Hanamichi Kawamura and the anonymous referee for valuable comments on the manuscript. 
\section{Preliminaries}\label{sec:Prelim}
In this section, we prepare tools which are used in the following sections.
\subsection{Algebraic setup}\label{subsec:alg_setup}
First we recall the notion of the harmonic algebra introduced in \cite{H1}.
Let $\frH^1\coloneqq\bQ+e_1\bQ\langle e_0,e_1\rangle \supset\frH^0\coloneqq\bQ+e_1\bQ\langle e_0,e_1\rangle e_0$, where $\bQ\langle e_0,e_1\rangle$ is a non-commutative polynomial algebra in two variables $e_0$ and $e_1$.
For a positive integer $k$, we set $e_k\coloneqq e_1e_0^{k-1}$.
We define the harmonic product $\ast$ on $\frH^1$ by $w\ast 1=1\ast w=w$, $e_{k_1}w_1\ast e_{k_2}w_2=e_{k_1}(w_1\ast e_{k_2}w_2)+e_{k_2}(e_{k_1}w_1\ast w_2)+e_{k_1+k_2}(w_1\ast w_2)$ ($w, w_1, w_2$ are words in $\frH^1$, $k_1,k_2\in\bZ_{>0}$) with $\bQ$-bilinearity.
We also define the shuffle product $\sh$ on $\bQ\langle e_0,e_1\rangle$ by $w\sh 1=1\sh w=w$, $u_1w_1\sh u_2w_2=u_1(w_1\sh u_2w_2)+u_2(u_1w_1\sh w_2)$ ($w, w_1, w_2$ are words in $\bQ\langle e_0,e_1\rangle$, $u_1,u_2\in\{e_0,e_1\}$) with $\bQ$-bilinearity.
Let $\bullet \in \{\ast,\sh\}$.
It is known that $\frH^1$ becomes a commutative $\bQ$-algebra with respect to the multiplication $\bullet$, which is denoted by $\frH^1_{\bullet}$.
The subspace $\frH^0$ of $\frH^1$ is closed under $\bullet$ and becomes a $\bQ$-subalgebra of $\frH^1_{\bullet}$, which is denoted by $\frH^0_{\bullet}$.
We define Muneta's shuffle product $\widetilde{\sh}$ on $\frH^1$ (\cite[\S3]{Mun}) by $w\widetilde{\sh} 1=1\widetilde{\sh} w=w$, $e_{k_1}w_1\widetilde{\sh} e_{k_2}w_2=e_{k_1}(w_1\widetilde{\sh} e_{k_2}w_2)+e_{k_2}(e_{k_1}w_1\widetilde{\sh} w_2)$ ($w, w_1, w_2$ are words in $\frH^1$, $k_1,k_2\in\bZ_{>0}$) with $\bQ$-bilinearity.

Next, we recall the harmonic (resp.~shuffle) regularized MZV introduced in \cite{IKZ}.
It is known that $\frH^1_{\bullet} \cong \frH^0_{\bullet}[e_1]$ as a $\bQ$-algebra (see \cite{H1} for $\bullet=*$ and \cite{Re} for $\bullet=\sh$).
Therefore, for $\bullet \in \{*,\sh\}$, any $a\in \frH^1_{\bullet}$ has a unique expression $a=\sum_{i=0}^{n} a_i \bullet e^{\bullet i}_1$, where $n\in \bZ_{\geq0}$, $a_i \in \frH^0_{\bullet}$ $(0 \leq i \leq n)$ and $e^{\bullet i}_1\coloneqq \overbrace{e_1 \bullet \cdots \bullet e_1}^{i}$.
By this expression, we define a $\bQ$-algebra homomorphism $\reg_{\bullet} \colon \frH^1_{\bullet} \cong \frH^0_{\bullet}[e_1] \rightarrow \frH^0_{\bullet}$ by $\reg_{\bullet}\left(\sum_{i=0}^n a_i \bullet e^{\bullet i}_1\right)\coloneqq a_0$.
We set $e_{\bk}\coloneqq e_{k_1}\cdots e_{k_r}$ for a non-empty index $\bk=(k_1,\dots, k_r)$ and $e_{\varnothing}\coloneqq 1$.
Then we define a $\bQ$-linear map $Z \colon \frH^0 \to \bR$ by $Z(e_{\bk})\coloneqq \zeta(\bk)$ for any admissible index $\bk$.
By using this terminology, we define the harmonic (resp.~ shuffle) regularized MZV $\zeta^{*}(\bk)$ (resp.~$\zeta^{\sh}(\bk)$) by $\zeta^{*}(\bk)=(Z\circ \reg_{*})(e_{\bk})$ (resp.~$\zeta^{\sh}(\bk) \coloneqq (Z \circ \reg_{\sh})(e_{\bk}))$ for any index $\bk$.

To calculate the shuffle regularized MZV, we use the following fact. 
\begin{lemma}[{Regularization formula, \cite[Proposition 8]{IKZ}}]\label{reg formula}
Let $w=w'e_0$ be an element of $\frH^0$ with $w'\in\frH^1$.
Then, for a non-negative integer $m$, we have
\[
\reg_{\sh}(we^m_1)=(-1)^m(w'\sh e^m_1)e_0.
\]
\end{lemma}

\subsection{Zagier's formulas for MZVs}
We quote some results on MZVs.
We use these results to evaluate some $\cS_1$-MZ(S)Vs and $\cS_2$-MZ(S)Vs. 
\begin{theorem}[{\cite[Theorem 1]{Z}}]\label{Zagier 1}
For non-negative integers $a$ and $b$, we have
\[
\zeta(\{2\}^a, 3, \{2\}^b)=2\sum_{r=1}^{a+b+1}(-1)^r\left\{\binom{2r}{2a+2}-\left(1-\frac{1}{2^{2r}}\right)\binom{2r}{2b+1}\right\}\zeta(\{2\}^{a+b-r+1})\zeta(2r+1),
\]
where $\{2\}^a$ denotes $a$ repetitions $\underbrace{2,\dots,2}_a$.
In particular, we have
\begin{multline}\label{Zagier 1 mod zeta(2)}
\zeta(\{2\}^a, 3, \{2\}^b)\\
\equiv2(-1)^{a+b+1}\left\{\binom{2a+2b+2}{2a+2}-\left(1-\frac{1}{4^{a+b+1}}\right)\binom{2a+2b+2}{2b+1}\right\}\zeta(2a+2b+3) \bmod{\zeta(2)}.
\end{multline}
\end{theorem}
\begin{theorem}[{\cite[Proposition 7]{Z}}]\label{Zagier 2}
Let $m$ and $n$ be positive integers with $n \geq2$ and $k\coloneqq m+n$ being odd. Define a positive integer $K$ as $k=2K+1$. Then we have
\begin{align*}
\zeta(m,n)=(-1)^m\sum_{s=0}^{K-1}\left\{\binom{k-2s-1}{m-1}+\binom{k-2s-1}{n-1}-\delta_{n, 2s}+(-1)^m\delta_{s, 0}\right\}\zeta(2s)\zeta(k-2s).
\end{align*}
Here $\delta_{x,y}$ is Kronecker's delta, and we understand $\zeta(0)=-\frac{1}{2}$. In particular, we have
\begin{align}\label{Zagier 2 mod zeta(2)}
\zeta(m, n)\equiv (-1)^{m+1}\frac{1}{2}\left\{\binom{k}{m}+(-1)^m\right\}\zeta(k) \bmod{\zeta(2)}.
\end{align}
\end{theorem}
\subsection{Double shuffle relation for $\cF_n$-MZVs}
The double shuffle relation (DSR) for $\cF_n$-MZVs with $\cF \in \{\cA, \cS\}$ established by Jarossay is a key tool in this paper.
We define $\bQ$-linear maps $Z_{\cA_n} \colon \frH^1 \to \cA_n$ and $Z_{\cS_n} \colon \frH^1 \to \overline{\cZ}\jump{t}/(t^n)$ by
\[
Z_{\cA_n}(e_{\bk})
\coloneqq \zeta^{}_{\cA_n}(\bk), \qquad Z_{\cS_n}(e_{\bk})
\coloneqq \zeta^{}_{\cS_n}(\bk)
\]
for any index $\bk$.
\begin{theorem}[{DSR for $\cF_n$-MZVs, \cite{J2}.~cf.~\cite[Theorems~1.3 and 1.9]{OSY}}]\label{DSRn}
For indices $\bk$ and $\bl=(l_1, \ldots, l_s)$ and a positive integer $n$, we have the harmonic relation for $\cF_n$-MZVs
\begin{equation}\label{eq:harmonic_relation}
Z_{\cF_n}(e_{\bk} *e_{\bl})=Z_{\cF_n}(e_{\bk})Z_{\cF_n}(e_{\bl})
\end{equation}
and the shuffle relation for $\cF_n$-MZVs
\begin{equation}\label{eq:shuffle_relation}
Z_{\cF_n}(e_{\bk} \sh e_{\bl})=(-1)^{\wt(\bl)}\sum_{\substack{\bl'=(l'_1, \ldots, l'_s) \in \bZ^s_{\geq0} \\ \wt(\bl') \leq n-1}}\left[\prod_{j=1}^s\binom{l_j+l'_j-1}{l'_j}\right]Z_{\cF_n}(e_{\bk}e_{\overline{\bl+\bl'}})x^{\wt(\bl')}_{\cF_n}.
\end{equation}
Here, we set $\wt(\bl')\coloneqq l'_1+\cdots+l'_s$,  $\overline{\bl+\bl'}\coloneqq (l_s+l_s', \dots, l_1+l_1')$ and
\[
x_{\cF_n}\coloneqq 
\begin{cases}
\pp_n\coloneqq (p \bmod{p^n})_p & \text{if $\cF=\cA$}, \\
t\bmod{t^n} & \text{if $\cF=\cS$}.
\end{cases}
\]
\end{theorem}
We refer to the case $\bk=\varnothing$ of the shuffle relation as the \emph{reversal formula}.

We also use the following relation for $\cF_n$-MZVs.
\begin{proposition}\label{antipode}
For an index $\bk=(k_1, \ldots, k_r)$, $n \in \bZ_{\ge1}$ and $\cF \in \{\cA, \cS\}$, we have
\begin{align*}
\sum_{i=0}^r(-1)^i\zeta^{}_{\cF_n}(k_1, \ldots, k_i)\zeta^{\star}_{\cF_n}(k_r, \ldots, k_{i+1})=0.
\end{align*}
\end{proposition}
\begin{proof}
This follows from the harmonic relation and \cite[Proposition 6]{IKOO} (Note that the sign of \cite[Proposition 6]{IKOO} is mistaken).
The case $\cF=\cA$ was first mentioned in \cite[Corollary~3.16 (42)]{SS}.
\end{proof}
\section{Special values}\label{sec:sp_value}
In this section, we explicitly evaluate some $\cF_n$-MZ(S)Vs.
For positive integers $n$ and $k$, set 
\begin{align*}
\frZ_{\cF_n}(k)\coloneqq 
\begin{cases}
\displaystyle\left(\frac{B_{p^{n-1}(p-1)-k+1}}{k-1+p^{n-1}}\bmod{p^n} \right)_p \in \cA_n & \text{if $\cF=\cA$}, \\
\zeta(k) \bmod{\zeta(2)} \in\overline{\cZ}\jump{t}/(t^n) & \text{if $\cF=\cS$}.
\end{cases}
\end{align*}
Here, $B_j$ is the $j$-th Seki--Bernoulli number and $\widehat{B}_j$ denotes $\frac{B_j}{j}$.
\begin{proposition}\label{prop:Z_[A_n]}
Let $n$ and $k$ be positive integers.
For $1\leq l\leq n-1$,
\[
\frZ_{\cA_n}(k+l)\pp_n^l=\sum_{j=1}^{n-l}(-1)^j\binom{n-l}{j}\left(\widehat{B}_{j(p-1)-k-l+1}\cdot p^l\bmod{p^n}\right)_p \in \cA_n
\]
holds.
In particular,
\[
\frZ_{\cA_2}(k+1)\pp_2=\left(\frac{B_{p-k-1}}{k+1}\cdot p\bmod{p^2}\right)_p\in \cA_2.
\]
\end{proposition}
\begin{proof}
Let $p$ be a sufficiently large prime number.
Then, by using the Kummer-type congruence proved by Zhi-Hong Sun \cite[Corollary~4.1]{Su}, we have
\begin{align*}
&\frac{B_{p^{n-1}(p-1)-k-l+1}}{k+l-1+p^{n-1}}\cdot p^l\\
&\equiv(-1)^{n+l}\sum_{j=1}^{n-l}(-1)^{j-1}\binom{p^{n-1}-1-j}{n-l-j}\binom{p^{n-1}-1}{j-1}\widehat{B}_{j(p-1)-k-l+1}\cdot p^l\pmod{p^{n}}.
\end{align*}
Since
\[
(-1)^{n+l-1}\binom{p^{n-1}-1-j}{n-l-j}\binom{p^{n-1}-1}{j-1} \equiv\binom{n-l}{j} \pmod{p^{n-1}},
\]
we have the desired formula.
\end{proof}
\subsection{Depth 1 case}
\begin{theorem}\label{thm:dep-1_general}
For positive integers $n$ and $k$, we have
\begin{align*}
\zeta^{}_{\cF_n}(k)=(-1)^k\sum_{l=1}^{n-1}\binom{k+l-1}{l}\frZ_{\cF_n}(k+l)x_{\cF_n}^l.
\end{align*}
\end{theorem}
\begin{proof}
The case $\cF=\cA$ is a special case of \cite[Theorem 1]{W}.
Nevertheless, we can state the direct proof as follows.
Let $p$ be a sufficiently large prime number.
By Euler's formula and Faulhaber's formula, we have
\begin{align*}
\sum_{m=1}^{p-1}\frac{1}{m^k} &\equiv\sum_{m=1}^{p-1}m^{\varphi(p^n)-k} \\ &\equiv \frac{1}{\varphi(p^n)-k+1}\sum_{l=1}^{n-1}\binom{\varphi(p^n)-k+1}{l}B_{\varphi(p^n)-k-l+1}\cdot p^l \\ &= -\sum_{l=1}^{n-1}\binom{\varphi(p^n)-k}{l}\frac{B_{\varphi(p^n)-k-l+1}}{k+l-1+p^{n-1}}\cdot p^l \pmod{p^n},
\end{align*}
where $\varphi$ is Euler's totient function.
By a simple congruence
\[
\binom{\varphi(p^n)-k}{l} \equiv (-1)^l\binom{k+l-1}{l}\pmod{p^{n-1}}
\]
and the fact that $B_j$ vanishes for odd $j\geq 3$, we have the desired equality in $\cA_n$.
Since the case $\cF=\cS$ is clear by definition, this completes the proof.
\end{proof}
\begin{remark}
By combining the case $\cF=\cA$ of Theorem~\ref{thm:dep-1_general} and Proposition~\ref{prop:Z_[A_n]}, we have
\[
\sum_{m=1}^{p-1}\frac{1}{m^k}\equiv(-1)^k\sum_{l=1}^{n-1}\binom{k+l-1}{l}\sum_{j=1}^{n-l}(-1)^j\binom{n-l}{j}\widehat{B}_{j(p-1)-k-l+1}p^l\pmod{p^n}
\]
for a sufficiently large prime $p$.
We can check that this holds for $p\geq n+k+1$.
This congruence is a generalization of \cite[Theorem~5.1~(a) and Remark~5.1]{Su} and \cite[Theorem~2.1]{Tau}.
However, the proof is identical to that put forward by Sun.
\end{remark}
\subsection{Depth 2 case}
Let $\tau_n\colon\overline{\cZ}\jump{t}\to\overline{\cZ}[t]$ be the truncation map defined by $\tau_n(\sum_{l=0}^{\infty}z_lt^l)\coloneqq\sum_{l=0}^{n-1}z_lt^l$ for a positive integer $n$.
In the following argument, we often identify $\zeta_{\cS_n}^{\bullet}(\bk)$ with $\tau_n(\zeta_{\widehat{\cS}}^{\bullet}(\bk))$, where $\bullet\in\{\varnothing, \star\}$.
Furthermore, we often abbreviate $\zeta(\bk) \bmod\zeta(2)$ (resp.~$\zeta^{\sh}(\bk)\bmod\zeta(2)$) to $\zeta(\bk)$ (resp.~$\zeta^{\sh}(\bk)$) in $\overline{\cZ}$.
\begin{theorem}\label{S_2 for k_1 k_2}
Let $k_1$ and $k_2$ be positive integers. Assume that $k\coloneqq k_1+k_2$ is even. Then we have 
\begin{align}
\zeta^{}_{\cF_2}(k_1, k_2)&=\frac{1}{2}\left\{(-1)^{k_1}k_2\binom{k+1}{k_1}-(-1)^{k_2}k_1\binom{k+1}{k_2}-k\right\}\frZ_{\cF_2}(k+1)x_{\cF_2}, \label{eq: S2 k_1 k_2}\\
\zeta^\star_{\cF_2}(k_1, k_2)&=\frac{1}{2}\left\{(-1)^{k_1}k_2\binom{k+1}{k_1}-(-1)^{k_2}k_1\binom{k+1}{k_2}+k\right\}\frZ_{\cF_2}(k+1)x_{\cF_2}. \label{eq: S2 k_1 k_2star}
\end{align}
\end{theorem}
\begin{proof}
The case $\cF=\cA$ was proved by Zhao, see \cite[Theorem 3.2]{Zh}.
Hereafter, we consider the case $\cF=\cS$.
First, we prove \eqref{eq: S2 k_1 k_2} for the case $k_1\geq2$. By the definition of $\zeta^{}_{\cS_2}(k_1,k_2)$, we have
\begin{align*} 
\zeta^{}_{\cS_2}(k_1, k_2)
=\zeta^{}_{\cS_1}(k_1, k_2)+\{k_2\zeta(k_2+1, k_1)+k_1\zeta(k_2, k_1+1)\}t.
\end{align*}
Since $k_1+k_2$ is even, we have $\zeta^{}_{\cS_1}(k_1, k_2)=0$ by definition.
Therefore, by using \eqref{Zagier 2 mod zeta(2)}, we obtain \eqref{eq: S2 k_1 k_2} for the case $k_1 \ge2$.
Next, we prove \eqref{eq: S2 k_1 k_2} for the case $k_1=1$ (then $k_2$ is odd).
We have
\begin{align}\label{by def}
\zeta^{}_{\cS_2}(1, k_2)=\{k_2\zeta^{\sh}(k_2+1, 1)+\zeta(k_2, 2)\}t.
\end{align}
By applying Theorem \ref{reg formula} for $w=e_1e^{k_2}_0$ and the sum formula for MZVs of depth~$2$, we have
\begin{equation}\label{calc of zeta sh}
\zeta^{\sh}(k_2+1, 1)=-\zeta(k_2, 2)-\cdots-\zeta(2, k_2)-2\zeta(1, k_2+1)=-\zeta(k_2+2)-\zeta(1, k_2+1).
\end{equation}
By \eqref{Zagier 2 mod zeta(2)}, we have
\begin{equation}\label{after Zagier}
\zeta(1, k_2+1)=\frac{k_2+1}{2}\zeta(k_2+2), \quad \zeta(k_2, 2)=\frac{1}{2}\left\{\frac{(k_2+2)(k_2+1)}{2}-1\right\}\zeta(k_2+2).
\end{equation}
From \eqref{by def}, \eqref{calc of zeta sh}, \eqref{after Zagier}, we obtain \eqref{eq: S2 k_1 k_2} for the case $k_1=1$. 
The formula \eqref{eq: S2 k_1 k_2star} follows from \eqref{eq: S2 k_1 k_2}, the fact $\zeta^\star_{\cS_2}(k_1, k_2)=\zeta^{}_{\cS_2}(k_1, k_2)+\zeta^{}_{\cS_2}(k_1+k_2)$, and $\zeta^{}_{\cS_2}(k)=(-1)^kk\zeta(k+1)t$ (Theorem~\ref{thm:dep-1_general} with $\cF_n=\cS_2$). 
\end{proof}
\subsection{Depth 3 case}
\begin{theorem}
Let $k_1, k_2, k_3$ be positive integers. Suppose that $k \coloneqq k_1+k_2+k_3$ is odd. Then we have
\begin{align*}
\zeta^{}_{\cF_1}(k_1, k_2, k_3)
=-\zeta^{\star}_{\cF_1}(k_1, k_2, k_3)
=\frac{1}{2}\left\{(-1)^{k_1}\binom{k}{k_1}-(-1)^{k_3}\binom{k}{k_3}\right\}\frZ_{\cF_1}(k).
\end{align*}
\end{theorem} 
\begin{proof}
The case $\cF=\cA$ was proved by Hoffman and Zhao; see \cite[Theorem~6.2]{H2} or \cite[Theorem~ 3.5]{Zh}.
Hereafter, we consider the case $\cF=\cS$. By Proposition \ref{antipode} and the reversal formula for $\cF_1$-MZVs, we have
\begin{equation}\label{eq:calc_dep3}
\zeta^{\star}_{\cF_1}(k_1,k_2,k_3)
=(-1)^{k_1+k_2+k_3}\zeta_{\cF_1}(k_1,k_2,k_3)
=-\zeta^{}_{\cF_1}(k_1,k_2,k_3).
\end{equation}
From
\begin{equation*}
\zeta^{\star}_{\cF_1}(k_1,k_2,k_3)
=\zeta^{}_{\cF_1}(k_1,k_2,k_3)
+\zeta^{}_{\cF_1}(k_1+k_2,k_3)
+\zeta^{}_{\cF_1}(k_1,k_2+k_3)
\end{equation*}
and the explicit formula for $\cF_1$-double zeta values \cite[(7.2), Example 9.4 (2)]{Kan}, we have
\begin{align} \label{eq:F1dep3}
\zeta^{}_{\cF_1}(k_1,k_2,k_3)
&=-\frac{\zeta^{}_{\cF_1}(k_1+k_2,k_3)+\zeta^{}_{\cF_1}(k_1,k_2+k_3)}{2}\\
&=-\frac{1}{2}
\left\{
(-1)^{k_3}\binom{k}{k_1+k_2}+(-1)^{k_2+k_3}\binom{k}{k_1}
\right\}
\frZ_{\cF_1}(k) \nonumber \\
&=\frac{1}{2}
\left\{
(-1)^{k_1}\binom{k}{k_1}-(-1)^{k_3}\binom{k}{k_3}
\right\}
\frZ_{\cF_1}(k). \nonumber
\end{align}
The formula for $\zeta^{\star}_{\cF_1}(k_1,k_2,k_3)$ is obtained by \eqref{eq:calc_dep3} and \eqref{eq:F1dep3}.
\end{proof}
\subsection{General depth case}
\begin{theorem}\label{rep F3}
For positive integers $r,k$ and for $\cF \in \{\cA, \cS\}$, we have
\begin{align}
\zeta^{}_{\cF_2}(\{k\}^r)&=(-1)^{r-1}k\frZ_{\cF_2}(rk+1)x_{\cF_2}, \label{rep.k non star} \\
\zeta^\star_{\cF_2}(\{k\}^r)&=k\frZ_{\cF_2}(rk+1)x_{\cF_2}. \label{rep.k star}
\end{align}
Moreover, we have
\begin{multline}\label{eq: rep.k F3}
\zeta^{}_{\cF_3}(\{k\}^r)=(-1)^{rk+r-1}\Biggl[k\frZ_{\cF_3}(rk+1)x_{\cF_3}+\\
\left\{\frac{k(rk+1)}{2}\frZ_{\cF_3}(rk+2)-k^2\sum_{l=1}^{r-1}\frZ_{\cF_3}(lk+1)\frZ_{\cF_3}((r-l)k+1)\right\}x_{\cF_3}^2\Biggr]
\end{multline}
and 
\begin{multline}\label{eq: rep.k F3 star}
\zeta^\star_{\cF_3}(\{k\}^r)=(-1)^{rk}\Biggl[k\frZ_{\cF_3}(rk+1)x_{\cF_3}+\\
\Biggl\{\frac{k(rk+1)}{2}\frZ_{\cF_3}(rk+2)
+k^2\sum_{l=1}^{r-1}\frZ_{\cF_3}(lk+1)\frZ_{\cF_3}((r-l)k+1)\Biggr\}x_{\cF_3}^2\Biggr].
\end{multline}
\end{theorem}
\begin{remark}
If $rk$ is odd, then $\frZ_{\cF_3}(rk+1)$ and $\frZ_{\cF_3}(lk+1)\frZ_{\cF_3}((r-l)k+1)$ are 0, and we have
\begin{align*}
\zeta^{}_{\cF_3}(\{k\}^r)=(-1)^r\frac{k(rk+1)}{2}\frZ_{\cF_3}(rk+2)x_{\cF_3}^2, \quad \zeta^\star_{\cF_3}(\{k\}^r)=-\frac{k(rk+1)}{2}\frZ_{\cF_3}(rk+2)x_{\cF_3}^2.
\end{align*}
These formulas for the case $\cF=\cA$ were first proved by Zhou and Cai in the last remark of \cite{ZC} but our proof differs from theirs.
\end{remark}
\begin{proof}
Since \eqref{rep.k non star} and \eqref{rep.k star} follows from \eqref{eq: rep.k F3} and \eqref{eq: rep.k F3 star} by taking modulo $x^2_{\cF_3}$, it is sufficient to prove \eqref{eq: rep.k F3} and \eqref{eq: rep.k F3 star}. 
Note that $(-1)^{rk}\frZ_{\cF_2}(rk+1)x_{\cF_2}=\frZ_{\cF_2}(rk+1)x_{\cF_2}$ holds because if $rk$ is odd, then $\frZ_{\cF_2}(rk+1)x_{\cF_2}=0$.

By Theorem~\ref{thm:dep-1_general} and the symmetric sum formula~\eqref{sym sum} proved in Section~\ref{sec:sum_formulas} with $\bk=(\{k\}^r)$, we have
\begin{equation}\label{F3 explicit}
\begin{split}
&r!\zeta^{}_{\cF_3}(\{k\}^r)\\
&=(-1)^{rk+r-1}(r-1)!\left\{rk\frZ_{\cF_3}(rk+1)x_{\cF_3}+\binom{rk+1}{2}\frZ_{\cF_3}(rk+2)x_{\cF_3}^2\right\}\\
&\quad +(-1)^{rk+r-2}\sum_{\substack{B_1 \sqcup B_2=\{1, \ldots, r\} \\ B_1,B_2\neq\varnothing}}(\#B_1-1)!(\#B_2-1)!b_1b_2\frZ_{\cF_3}(b_1+1)\frZ_{\cF_3}(b_2+1)x_{\cF_3}^2,
\end{split}
\end{equation}
where $b_1=b_1(\{k\}^r)$ and $b_2=b_2(\{k\}^r)$ are defined as in Theorem~\ref{sym sum non star}.
Set $l\coloneqq \#B_1$. Then we see that $1\leq l \leq r-1$, $\#B_2=r-l$, $b_1=lk$ and $b_2=(r-l)k$. Moreover, the number of ways of dividing $\{1, \ldots, r\}$ into two non-empty subsets $B_1$ and $B_2$ with $\#B_1=l$ is just $\binom{r}{l}$. Therefore, the summation for the partition in the right-hand side of \eqref{F3 explicit} coincides with 
\begin{align*}
&\sum_{l=1}^{r-1}\binom{r}{l}(l-1)!(r-l-1)!\cdot l(r-l)k^2\frZ_{\cF_3}(lk+1)\frZ_{\cF_3}((r-l)k+1)x_{\cF_3}^2\\
&=k^2\cdot r!\sum_{l=1}^{r-1}\frZ_{\cF_3}(lk+1)\frZ_{\cF_3}((r-l)k+1)x_{\cF_3}^2.
\end{align*}
Thus we obtained \eqref{eq: rep.k F3}.
The formula \eqref{eq: rep.k F3 star} is obtained in the same manner.
\end{proof}
\begin{theorem}\label{S1 1 2 1}
For non-negative integers $a$ and $b$, we have
\begin{align}
\zeta^{}_{\cF_1}(\{1\}^a, 2, \{1\}^b)
=(-1)^b\binom{a+b+2}{a+1}\frZ_{\cF_1}(a+b+2), \label{eq: S1 121}\\
\zeta^\star_{\cF_1}(\{1\}^a, 2, \{1\}^b)
=(-1)^b\binom{a+b+2}{a+1}\frZ_{\cF_1}(a+b+2). \label{eq: S1 121star}
\end{align}
\end{theorem}
\begin{proof}
The case $\cF=\cA$ was proved by Hessami-Pilehrood--Hessami-Pilehrood--Tauraso; see \cite[Theorem~4.5]{HHT}.
Hereafter, we consider the case $\cF=\cS$. 
By the definition of $\zeta^{}_{\cS_1}(\bk)$ and the fact that  $\zeta^{\sh}(\{1\}^k)=0$ for $k\geq 1$, we have
\begin{align}\label{by def for a b}
\zeta^{}_{\cS_1}(\{1\}^a, 2, \{1\}^b)=\zeta^{\sh}(\{1\}^a, 2, \{1\}^b)+(-1)^{a+b}\zeta^{\sh}(\{1\}^b, 2, \{1\}^a).
\end{align}
Applying Lemma~\ref{reg formula} for $w=e^{a+1}_1e_0$ and $m=b$, and using the duality for MZVs, we have
\begin{align}\label{reg form for a b}
\zeta^{\sh}(\{1\}^a, 2, \{1\}^b)=(-1)^b\binom{a+b+1}{b}\zeta(a+b+2).
\end{align}
From \eqref{by def for a b} and \eqref{reg form for a b}, we obtain \eqref{eq: S1 121}.
The formula \eqref{eq: S1 121star} follows from \eqref{eq: S1 121}, Proposition \ref{antipode} and the fact that $\zeta^{}_{\cS_1}(\{1\}^r)=\zeta^{\star}_{\cS_1}(\{1\}^r)=0$ for $r\geq 1$.
The last fact is well-known and a special case of Theorem~\ref{rep F3}.
\end{proof}
\begin{remark}
We can also prove \eqref{eq: S1 121star} using the Hoffman duality (\cite[Theorem~4.6]{H2} and \cite[Corollarie~1.12]{J1}) and the explicit formula for $\zeta^{}_{\cF_1}(a+1, b+1)$. 
\end{remark}
\begin{theorem}\label{S1 for 2 3 2}
For non-negative integers $a$ and $b$, we have
\begin{align}
\zeta^{}_{\cF_1}(\{2\}^a, 3, \{2\}^b)&=\frac{(-1)^{a+b}2(a-b)}{a+1}\binom{2a+2b+3}{2b+2}\frZ_{\cF_1}(2a+2b+3),\label{S1 for 2 3 2 non star} \\
\zeta^\star_{\cF_1}(\{2\}^a, 3, \{2\}^b)&=\frac{2(b-a)}{a+1}\binom{2a+2b+3}{2b+2}\frZ_{\cF_1}(2a+2b+3). \label{S1 for 2 3 2 star}
\end{align}
\end{theorem}
\begin{proof}
The case $\cF=\cA$ was proved by Hessami-Pilehrood--Hessami-Pilehrood--Tauraso; see \cite[Theorem~4.1]{HHT}.
Hereafter, we consider the case $\cF=\cS$. 
By the definition of the $\cS_1$-MZV and the fact that  $\zeta(\{2\}^r) \equiv 0 \pmod{\zeta(2)}$ for $r\geq1$, we have
\begin{align*}
\zeta^{}_{\cS_1}(\{2\}^a, 3, \{2\}^b)=\zeta(\{2\}^a, 3, \{2\}^b)-\zeta(\{2\}^b, 3, \{2\}^a).
\end{align*}
Thus we obtain \eqref{S1 for 2 3 2 non star} by the formula~\eqref{Zagier 1 mod zeta(2)} and straightforward calculation of binomial coefficients. 
The formula \eqref{S1 for 2 3 2 star} is obtained by \eqref{S1 for 2 3 2 non star}, Proposition \ref{antipode} and a special case of Theorem~\ref{rep F3}, that is, the fact that $\zeta^{}_{\cS_1}(\{2\}^r)=\zeta^{\star}_{\cS_1}(\{2\}^r)=0$ for $r\geq 1$.
\end{proof}
\begin{theorem}\label{S_1 for 2 1 2}
For non-negative integers $a$ and $b$, we have
\begin{align}
\zeta^{}_{\cF_1}(\{2\}^a, 1, \{2\}^b)&=4(-1)^{a+b}\frac{a-b}{2a+1}\left(1-\frac{1}{4^{a+b}}\right)\binom{2a+2b+1}{2b+1}\frZ_{\cF_1}(2a+2b+1), \label{eq: S1 212}\\
\zeta^\star_{\cF_1}(\{2\}^a, 1, \{2\}^b)&=\frac{4(b-a)}{2a+1}\left(1-\frac{1}{4^{a+b}}\right)\binom{2a+2b+1}{2b+1}\frZ_{\cF_1}(2a+2b+1). \label{eq: S1 212star}
\end{align}
\end{theorem}
\begin{proof}
The case $\cF=\cA$ was proved by Hessami-Pilehrood--Hessami-Pilehrood--Tauraso; see \cite[Theorem~4.2]{HHT}.
Hereafter, we consider the case $\cF=\cS$. 
First, we prove \eqref{eq: S1 212} for the case $a, b \geq1$. By the definition of the $\cS_1$-MZV and the duality for MZVs, we obtain
\begin{align*}
\zeta^{}_{\cS_1}(\{2\}^a, 1, \{2\}^b)=\zeta(\{2\}^{b-1}, 3, \{2\}^a)-\zeta(\{2\}^{a-1}, 3, \{2\}^b).
\end{align*}
Then we obtain \eqref{eq: S1 212} by a similar calculation in Theorem~\ref{S1 for 2 3 2} using \eqref{Zagier 1 mod zeta(2)}. 

Next we prove \eqref{eq: S1 212} for the case $a\geq1$ and $b=0$. We have
\begin{equation}
\zeta^{}_{\cS_1}(\{2\}^a, 1)=\zeta^{\sh}(\{2\}^a, 1)-\zeta(1, \{2\}^a).
\end{equation}
Applying Lemma \ref{reg formula} for $w=(e_1e_0)^a$ and $m=1$, we obtain
\begin{align*}
\zeta^{\sh}(\{2\}^a, 1)=-2\sum_{j=0}^{a-1}\zeta(\{2\}^j, 1, \{2\}^{a-j}).
\end{align*}
Thus, by the duality for MZVs, we have
\begin{align}\label{by duality}
\zeta^{}_{\cS_1}(\{2\}^a, 1)&=-\zeta(1, \{2\}^a)-2\sum_{j=0}^{a-1}\zeta(\{2\}^j, 1, \{2\}^{a-j})\\
&=-\zeta(\{2\}^{a-1}, 3)-2\sum_{j=0}^{a-1}\zeta(\{2\}^{a-j-1}, 3, \{2\}^j). \nonumber
\end{align}
By \eqref{Zagier 1 mod zeta(2)}, we obtain
\begin{equation}\label{by Zagier}
\zeta(\{2\}^{a-j-1}, 3, \{2\}^j)= 2(-1)^a\left\{\binom{2a}{2j}-\left(1-\frac{1}{4^a}\right)\binom{2a}{2j+1}\right\}\zeta(2a+1)
\end{equation}
for $0 \leq j \leq a-1$. Therefore, from \eqref{by duality} and \eqref{by Zagier}, we obtain \eqref{eq: S1 212} for the case $a\geq1$ and $b=0$.
 The case $a=0$ and $b\geq1$ of \eqref{eq: S1 212} follows easily from the reversal formula and \eqref{by duality} with $a\geq1$.
 This completes the proof of \eqref{eq: S1 212}.
 The formula \eqref{eq: S1 212star} is obtained by \eqref{eq: S1 212}, Proposition~\ref{antipode} and the fact that $\zeta^{}_{\cS_1}(\{2\}^r)=\zeta^{\star}_{\cS_1}(\{2\}^r)=0$ for $r\geq 1$.
\end{proof}
\begin{remark}
Tasaka and Yamamoto proved an analogous formula of Theorem~\ref{Zagier 1} for $\zeta^{\star}(\{2\}^a, 1, \{2\}^b)$; see \cite[Theorem~1.6]{TY}.
We can also obtain Theorem~\ref{S_1 for 2 1 2} by a similar approach using \cite[Theorem~1.6]{TY} and Proposition \ref{antipode} instead of Zagier's formula (Theorem~\ref{Zagier 1}).
\end{remark}
The following theorem is a refinement of the even weight case in Theorem~\ref{S1 1 2 1}.
\begin{theorem}\label{S2 1 2 1}
Let $a$ and $b$ be non-negative integers. Assume that $a+b$ is even. Then we have
\begin{align}
\zeta^{}_{\cF_2}(\{1\}^a, 2, \{1\}^b)=\frac{1}{2}\left\{1+(-1)^a\binom{a+b+3}{b+2}\right\}\frZ_{\cF_2}(a+b+3)x_{\cF_2}, \label{S2 1 2 1 non star}\\
\zeta^\star_{\cF_2}(\{1\}^a, 2, \{1\}^b)=\frac{1}{2}\left\{1+(-1)^a\binom{a+b+3}{a+2}\right\}\frZ_{\cF_2}(a+b+3)x_{\cF_2}. \label{S2 1 2 1 star}
\end{align}
\end{theorem}
\begin{proof}
The case $\cF=\cA$ was proved by Hessami-Pilehrood--Hessami-Pilehrood--Tauraso \cite[Theorem~4.5]{HHT}; there is also another proof by Sakugawa and the third author \cite[Theorem~3.18]{SS}.
Hereafter, we consider the case $\cF=\cS$. 
We first prove the formula \eqref{S2 1 2 1 non star}. Set $d\coloneqq a+b+1$ and $\bk=(k_1, \ldots, k_d)\coloneqq (\{1\}^a, 2, \{1\}^b)$.
For $0\leq i \leq d$, we set
\begin{align*}
P_i(t)&\coloneqq (-1)^{k_{i+1}+\cdots+k_d}\zeta^{\sh}(k_1, \ldots, k_i)\\
&\quad\times\sum_{\substack{l_{i+1}, \ldots, l_d \geq0 \\ l_{i+1}+\cdots+l_d \leq1}}\left[\prod_{j=i+1}^d\binom{k_j+l_j-1}{l_j}\right]\zeta^{\sh}(k_d+l_d, \ldots, k_{i+1}+l_{i+1})t^{l_{i+1}+\cdots+l_d}.
\end{align*}
Note that this expression for the case $i=d$ means $P_d(t)=\zeta^{\sh}(\bk)$.
Then, by the definition of $\zeta^{}_{\cS_2}(\bk)$, we have
\[
\zeta^{}_{\cS_2}(\bk)=\sum_{i=0}^dP_i(t).
\]
We prove that for $1\leq i\leq a+b+1$, $P_i(t)=0$ in $\overline{\cZ}[t]$.
Since $\zeta^{\sh}(\{1\}^k)=0$ for a positive integer $k$, we obtain $P_1(t)=\cdots=P_a(t)=0$.
For $0 \le j \le b$, we calculate $P_{a+j+1}(t)$. By the definition of $P_i(t)$, we have
\[
P_{a+j+1}(t)=(-1)^{b-j}\zeta^{\sh}(\{1\}^a, 2, \{1\}^j)\left(\zeta^{\sh}(\{1\}^{b-j})+\sum_{i=1}^{b-j}\zeta^{\sh}(\{1\}^{i-1}, 2, \{1\}^{b-j-i})t\right).
\]
By \eqref{reg form for a b}, we have $\zeta^{\sh}(\{1\}^a, 2, \{1\}^j)=(-1)^j\binom{a+j+1}{j}\zeta(a+j+2)$. 
Since $\zeta(a+b+2)=0$ in $\overline{\cZ}$, we have $P_{a+b+1}(t)=0$ in $\overline{\cZ}[t]$.
Assume that $j<b$.
In this case, $\zeta^{\sh}(\{1\}^{b-j})=0$ holds.
By the shuffle-regularized sum formula \cite[Lemma~3.3]{Li} (or \cite[Theorem~1.2]{KS}), and the duality for MZVs, we have
\begin{align*}
\sum_{i=1}^{b-j}\zeta^{\sh}(\{1\}^{i-1}, 2, \{1\}^{b-j-i})=(-1)^{b-j-1}\zeta(b-j+1).
\end{align*}
Since $a+b$ is even, $a+j+2 \not \equiv b-j+1 \pmod{2}$ and thus we have $P_{a+j+1}(t)=0$ in $\overline{\cZ}$. 

The calculation of $P_0(t)$ remains. State $P_0(t)=A+Bt$ with $A, B\in \overline{\cZ}$. Then from \eqref{reg form for a b}, we have $A=\zeta^{\sh}(\{1\}^b, 2, \{1\}^a)=0$ in $\overline{\cZ}$.
By definition, $B$ is expressed as follows:
\begin{align*}
B
&=\sum_{\substack{l+m=b-1\\ l, m\geq0}}\zeta^{\sh}(\{1\}^l, 2, \{1\}^m, 2, \{1\}^a)\\
&\quad +2\zeta^{\sh}(\{1\}^b, 3, \{1\}^a)+\sum_{\substack{m+n=a-1\\ m, n\geq0}}\zeta^{\sh}(\{1\}^b, 2, \{1\}^m, 2, \{1\}^n).
\end{align*}
For general non-negative integers $l,m$ and $n$, by the regularization formula (Lemma~\ref{reg formula}), we have
\begin{align*}
\zeta^{\sh}(\{1\}^l, 2, \{1\}^m, 2, \{1\}^n)
&=(-1)^n\sum_{\substack{r+s=n \\ r, s \geq0}}\binom{r+l+1}{r}\binom{s+m+1}{s}\zeta(\{1\}^{r+l}, 2, \{1\}^{s+m}, 2),\\
\zeta^{\sh}(\{1\}^l, 3, \{1\}^n)&=(-1)^n\sum_{\substack{r+s=n \\ r, s \geq0}}\binom{r+l+1}{r}\zeta(\{1\}^{r+l}, 2, \{1\}^{s-1}, 2),
\end{align*}
where $\zeta(\{1\}^{r+l}, 2, \{1\}^{-1}, 2)$ means $\zeta(\{1\}^{r+l}, 3)$.
Thus, with the duality for MZVs, we obtain
\begin{align*}
B=&\sum_{\substack{l+m=b-1 \\ l, m\geq0}}(-1)^a\sum_{\substack{r+s=a\\r, s\geq0}}\binom{r+l+1}{r}\binom{s+m+1}{s}\zeta(s+m+2, r+l+2)\\
&+2(-1)^a\sum_{\substack{r+s=a \\ r, s\geq0}}\binom{r+b+1}{r}\zeta(s+1, r+b+2)\\
&+\sum_{\substack{m+n=a-1\\m,n\geq0}}(-1)^n\sum_{\substack{r+s=n \\ r, s \geq0}}\binom{r+b+1}{r}\binom{s+m+1}{s}\zeta(s+m+2, r+b+2).
\end{align*}
Since $a+b+3$ is odd, by using \eqref{Zagier 2 mod zeta(2)}, we can rewrite $B$ as a rational multiple of the Riemann zeta value $\zeta(a+b+3)$.
Specifically, we have $B=\frac{1}{2}C\zeta(a+b+3)$ with
\begin{align*}
C=&\sum_{\substack{l+m=b-1 \\ l, m\geq0}}(-1)^a\sum_{\substack{r+s=a\\r, s\geq0}}\binom{r+l+1}{r}\binom{s+m+1}{s}(-1)^{s+m+1}\left\{\binom{a+b+3}{s+m+2}+(-1)^{s+m}\right\}\\
&+2(-1)^a\sum_{\substack{r+s=a \\ r, s\geq0}}\binom{r+b+1}{r}(-1)^{s}\left\{\binom{a+b+3}{s+1}+(-1)^{s+1}\right\}\\
&+\sum_{\substack{m+n=a-1\\m,n\geq0}}(-1)^n\sum_{\substack{r+s=n \\ r, s \geq0}}\binom{r+b+1}{r}\binom{s+m+1}{s}(-1)^{s+m+1}\left\{\binom{a+b+3}{s+m+2}+(-1)^{s+m}\right\}.
\end{align*}
Therefore, it suffices to prove the following:
\begin{equation}\label{key binom calc}
  C=1+(-1)^a\binom{a+b+3}{b+2}.
\end{equation}
We prove this in the Appendix. From this, we obtain the desired formula for $\zeta^{}_{\cS_2}(\{1\}^a, 2, \{1\}^b)$. 

Next, we prove \eqref{S2 1 2 1 star}.
By Proposition~\ref{antipode}, we have
\begin{align*}
\zeta^{\star}_{\cS_2}(\{1\}^a,2,\{1\}^b)-\zeta^{}_{\cS_2}(\{1\}^b,2,\{1\}^a)&=\sum_{j=1}^a(-1)^j\zeta^{}_{\cS_2}(\{1\}^b, 2, \{1\}^{a-j})\zeta^\star_{\cS_2}(\{1\}^j)\\
&\quad+\sum_{i=1}^b(-1)^{b-i}\zeta^{}_{\cS_2}(\{1\}^{b+1-i})\zeta^\star_{\cS_2}(\{1\}^a, 2, \{1\}^{i-1}).
\end{align*}
It is sufficient to show that the right-hand side vanishes.
If $j$ is odd, then we have $\zeta^\star_{\cS_2}(\{1\}^j)=0$ by \eqref{rep.k star}.
If $j$ is even, then both $\zeta^\star_{\cS_2}(\{1\}^j)$ and $\zeta^{}_{\cS_2}(\{1\}^b, 2, \{1\}^{a-j})$ can be seen as elements of $t\overline{\cZ}[t]$ by \eqref{rep.k star} and \eqref{eq: S1 121}. 
Thus the first summation vanishes in $\overline{\cZ}\jump{t}/(t^2)$. 
Similarly, if $b-i$ is even, then we have $\zeta^{}_{\cS_2}(\{1\}^{b+1-i})=0$ by \eqref{rep.k non star}.
If $b-i$ is odd, then both $\zeta^{}_{\cS_2}(\{1\}^{b+1-i})$ and $\zeta^\star_{\cS_2}(\{1\}^a, 2, \{1\}^{i-1})$ can be seen as elements of $t\overline{\cZ}[t]$ by \eqref{rep.k non star} and \eqref{eq: S1 121star}. Therefore, the second summation also vanishes.
\end{proof}
\begin{remark}
The proof of the case $\cF=\cA$ of Theorem~\ref{S2 1 2 1} by Sakugawa and the third author is based on the `$\cA_2$-duality' \cite[Remark~3.14 (40)]{SS}.
If the `$\cS_2$-duality' is established, then we can obtain another proof of the case $\cF=\cS$ of Theorem~\ref{S2 1 2 1}.
When we were writing this paper, a preprint \cite{TT} by Takeyama and Tasaka appeared on arXiv.
Their \cite[Corollary~6.8]{TT} contains the $\cS_2$-duality as a special case.
\end{remark}
\section{Bowman--Bradley type theorem}\label{sec:BB}
Murahara, Onozuka and the third author \cite{MOS} proved the Bowman--Bradley type theorem for $\cA_2$-MZ(S)Vs (= the case  $\cF=\cA$ of Theorem~\ref{BB thm S_2}).
In this section, we prove the $\cS_2$-counterpart of their theorem.
By combining these two theorems, we have the following.
\begin{theorem}[{Bowman--Bradley type theorem for $\cF_2$-MZ(S)V}]\label{BB thm S_2}
For non-negative integers $l$ and $m$ with $(l, m)\neq(0, 0)$, we have
\begin{equation}\label{BB SMZV}
\begin{split}
&\sum_{\substack{m_0+\cdots+m_{2l}=m \\ m_0, \ldots, m_{2l} \geq0}}\zeta^{}_{\cF_2}\bigl(\{2\}^{m_0}, 1, \{2\}^{m_1}, 3, \{2\}^{m_2}, \dots, \{2\}^{m_{2l-2}}, 1, \{2\}^{m_{2l-1}}, 3, \{2\}^{m_{2l}}\bigr) \\
&=(-1)^m\left\{(-1)^l2^{1-2l}\binom{l+m}{l}-4\binom{2l+m}{2l}\right\}\frZ_{\cF_2}(4l+2m+1)x_{\cF_2},
\end{split}
\end{equation}
\begin{equation}\label{BB SMZSV}
\begin{split}
&\sum_{\substack{m_0+\cdots+m_{2l}=m \\ m_0, \ldots, m_{2l} \geq0}}\zeta^\star_{\cF_2}\bigl(\{2\}^{m_0}, 1, \{2\}^{m_1}, 3, \{2\}^{m_2}, \dots, \{2\}^{m_{2l-2}}, 1, \{2\}^{m_{2l-1}}, 3, \{2\}^{m_{2l}}\bigr) \\
&=(-1)^l2^{1-2l}\binom{l+m}{l}\frZ_{\cF_2}(4l+2m+1)x_{\cF_2}.
\end{split}
\end{equation}
\end{theorem}
This gives a partial lift of the Bowman--Bradley type theorem for $\cF_1$-MZ(S)V proved by Saito and Wakabayashi \cite{SW2}.
Note that the proof of the $\cF=\cS$ case in Theorem~\ref{BB thm S_2} presented here is essentially the same as the proof of the $\cF=\cA$ case by Murahara, Onozuka and the third author \cite{MOS}.
In contrast, proofs of some sum formulas which will be given in the next section are different from those in the previous study.

We prepare two lemmas for the proof.
\begin{lemma}\label{2 sh 2}
For non-negative integers $l$ and $m$ with $(l, m)\neq (0, 0)$, we have
\begin{align*}
Z_{\cS_2}\bigl(e^{l+m}_2\sh e^l_2\bigr)=(-1)^m2\left\{1-2\binom{4l+2m}{2l}\right\}\zeta(4l+2m+1)t.
\end{align*}
\end{lemma}
\begin{proof}
This lemma is the $\cS_2$-counterpart of \cite[Lemma 2.5]{MOS} and is proved from the same argument in \cite{MOS} by using the explicit evaluation of $\zeta^{}_{\cS_2}(\{2\}^r)$ (\eqref{rep.k non star} with $k=2$), \eqref{S1 for 2 3 2 non star} and \eqref{eq:shuffle_relation} with $\cF_n=\cS_2$.
\end{proof}

For a positive integer $n$, define a $\bQ$-linear map $Z^{\star}_{\cS_n} \colon \frH^1 \to \overline{\cZ}\jump{t}/(t^n)$ by $Z^{\star}_{\cS_n}(e_{\bk}) \coloneqq \zeta^{\star}_{\cS_n}(\bk)$ for any index $\bk$.
\begin{lemma}\label{Yamamoto}
For non-negative integers $l$ and $m$, we have
\begin{align*}
&Z^{\star}_{\cS_2}\bigl((e_1e_3)^l \; \widetilde{\sh}\; e^m_2\bigr)\\
&=\sum_{\substack{2i+k+u=2l \\ j+n+v=m}}(-1)^{j+k}\binom{k+n}{k}\binom{u+v}{u}Z_{\cS_2}\bigl((e_1e_3)^i \widetilde{\sh} \; e^j_2\bigr) Z^{\star}_{\cS_2}\bigl(e^{k+n}_2\bigr) Z^{\star}_{\cS_2}\bigl(e^{u+v}_2\bigr).
\end{align*}
\end{lemma}
\begin{proof}
This follows immediately from \cite[equation (3.1)]{Y} and \eqref{eq:harmonic_relation} with $\cF_n=\cS_2$.
\end{proof}
\begin{proof}[Proof of the $\cF=\cS$ case in Theorem \ref{BB thm S_2}]
We prove \eqref{BB SMZV} by induction on $l \geq0$.
The case $l=0$ holds by the explicit evaluation of $\zeta^{}_{\cS_2}(\{2\}^r)$ (\eqref{rep.k non star} with $k=2$). Let $l$ be a positive integer and $m$ a non-negative integer. By \cite[Lemma 2.1]{MOS}, we have
\begin{align*}
&Z_{\cS_2}\bigl((e_1e_3)^l \; \widetilde{\sh} \; e^m_2\bigr)\\
&=4^{-l}Z_{\cS_2}\bigl(e^{l+m}_2 \sh e^l_2\bigr)-\sum_{k=0}^{l-1}4^{k-l}\binom{2l+m-2k}{l-k}Z_{\cS_2}\bigl((e_1e_3)^k\; \widetilde{\sh} \; e^{2l+m-2k}_2\bigr).
\end{align*}
Hence, by Lemma \ref{2 sh 2} and the induction hypothesis, we have
\begin{align*}
&Z_{\cS_2}\bigl((e_1e_3)^l \; \widetilde{\sh} \; e^m_2\bigr)\\
&=(-1)^m2^{1-2l}\left\{1-2\binom{4l+2m}{2l}\right\}\zeta(4l+2m+1)t\\
&\quad-\sum_{k=0}^{l-1}4^{k-l}\binom{2l+m-2k}{l-k}\\
&\qquad \cdot (-1)^m\left\{(-1)^k2^{1-2k}\binom{2l+m-k}{k}-4\binom{2l+m}{2k}\right\}\zeta(4l+2m+1)t\bmod{\zeta(2)}.
\end{align*}
We see that this coincides with the desired formula by using \cite[Lemma 2.6]{MOS}. We also obtain \eqref{BB SMZSV} by the same argument in \cite{MOS} using \eqref{BB SMZV}, \eqref{rep.k star} with $k=2$ and Lemma~\ref{Yamamoto}.
\end{proof}
%
\section{Sum formulas}\label{sec:sum_formulas}
In this section, we prove the \emph{$\cF_n$-symmetric sum formula} (= Theorem~\ref{sym sum}), the \emph{$\cF_n$-sum formula over $I_{k,r}$} for $n=2,3$ (= Theorem~\ref{sum formula F3}), and the \emph{$\cF_2$-sum formula over $I_{k,r,i}$} (= Theorem~\ref{i adm sum formula}).
\subsection{$\cF_n$-symmetric sum formula}
We first state the $\cF_n$-symmetric sum formula.
\begin{theorem}[{$\cF_n$-symmetric sum formula}]\label{sym sum}
Let $n$ and $r$ be positive integers and $\bk=(k_1,\dots,k_r)$ an index.
Then, we have
\begin{align}
\sum_{\sigma \in \frS_r}\zeta^{}_{\cF_n}(\sigma(\bk))&=\sum_{\cB=\{B_1, \ldots, B_l\}}(-1)^{r-l}c(\cB)\zeta^{}_{\cF_n}(b_1(\bk))\cdots\zeta^{}_{\cF_n}(b_l(\bk)),\label{sym sum non star}\\
\sum_{\sigma \in \frS_r}\zeta^\star_{\cF_n}(\sigma(\bk))&=\sum_{\cB =\{B_1, \ldots, B_l\}}c(\cB)\zeta^{}_{\cF_n}(b_1(\bk))\cdots\zeta^{}_{\cF_n}(b_l(\bk)).\label{sym sum star}
\end{align}
Here, $\frS_r$ denotes the symmetric group of degree $r$.  For $\sigma \in \frS_r$, set $\sigma(\bk)\coloneqq (k_{\sigma(1)}, \ldots, k_{\sigma(r)})$. $\cB=\{B_1, \ldots, B_l\}$ runs all partitions of $\{1, \ldots, r\}$, that is, $\cB=\{B_1, \ldots, B_l\}$ satisfies that $\{1, \ldots, r\}=\bigsqcup_{i=1}^l B_i$ and $B_i \neq \varnothing \; (1\leq i \leq l)$. Moreover, we set $c(\cB)\coloneqq (\#B_1-1)!\cdots(\#B_l-1)!$ and $b_i(\bk)\coloneqq \sum_{j \in B_i}k_j$.
\end{theorem}
\begin{proof}
Since $\cF_n$-MZVs satisfy the harmonic relation \eqref{eq:harmonic_relation}, we see that the desired formulas hold by the same argument as \cite[Theorem 4.1]{H2}.
\end{proof}
\subsection{$\cF_n$-sum formula over $I_{k,r}$ for $n=2,3$}
Next, we prove the $\cF_n$-sum formula over $I_{k,r}$ for $n=2, 3$. For positive integers $n,r$ and $k$ with $r\leq k$, $\cF \in \{\cA, \cS\}$ and $\bullet \in \{\varnothing, \star\}$, set
\begin{align*}
S^{\bullet}_{\cF_n; k, r}\coloneqq \sum_{\bk \in I_{k, r}}\zeta^\bullet_{\cF_n}(\bk),
\end{align*}
where $I_{k,r}$ denotes the set of all indices $\bk$ with $\wt(\bk)=k$ and $\dep(\bk)=r$.

\begin{theorem}[{$\cF_n$-sum formula over $I_{k,r}$ for $n=2,3$}]\label{sum formula F3}
For positive integers $r$ and $k$ with $r\leq k$, we have
\begin{align}\label{eq: SF of F2}
S_{\cF_2; k, r}=(-1)^{r-1}\binom{k}{r}\frZ_{\cF_2}(k+1)x_{\cF_2}, \quad S^{\star}_{\cF_2; k, r}=\binom{k}{r}\frZ_{\cF_2}(k+1)x_{\cF_2}.
\end{align}
Moreover, we have
\begin{equation}\label{eq: SF of F3}
S_{\cF_3; k, r}
=(-1)^{k+r-1}\Biggl[\binom{k}{r}\frZ_{\cF_3}(k+1)x_{\cF_3}+\Biggl\{\frac{k+1}{2}\binom{k}{r}\frZ_{\cF_3}(k+2)-\frac{1}{r!}\cdot T_{k,r}\Biggr\}x_{\cF_3}^2\Biggr]
\end{equation}
and
\begin{equation}\label{eq: SF of F3 star}
S^{\star}_{\cF_3; k, r}
=(-1)^{k}\Biggl[\binom{k}{r}\frZ_{\cF_3}(k+1)x_{\cF_3}+\Biggl\{\frac{k+1}{2}\binom{k}{r}\frZ_{\cF_3}(k+2)+\frac{1}{r!}\cdot T_{k,r}\Biggr\}x_{\cF_3}^2\Biggr],
\end{equation}
where
\[
T_{k,r}=\sum_{\substack{B_1\sqcup B_2=\{1,\dots,r\} \\ B_1, B_2\neq \varnothing}}\sum_{\substack{b_1+b_2=k \\ b_1\geq \#B_1, b_2\geq \#B_2}}(b_1)_{\#B_1}(b_2)_{\#B_2}\cdot \frZ_{\cF_3}(b_1+1)\frZ_{\cF_3}(b_2+1)
\]
and the symbol $(n)_m$ denotes $n(n-1)\cdots (n-m+1)$.
\end{theorem}
\begin{remark}
If $k=b_1+b_2$ is odd, since $\frZ_{\cF_3}(k+1)$ and $\frZ_{\cF_3}(b_1+1)\frZ_{\cF_3}(b_2+1)$ are $0$, we have
\begin{align*}
S_{\cF_3; k, r}=(-1)^r\frac{k+1}{2}\binom{k}{r}\frZ_{\cF_3}(k+2)x_{\cF_3}^2, \quad S^\star_{\cF_3; k, r}=-\frac{k+1}{2}\binom{k}{r}\frZ_{\cF_3}(k+2)x_{\cF_3}^2,
\end{align*}
which were first proved by the third author and Yamamoto \cite[Theorem 2.5]{SY} for $\cF=\cA$.
\end{remark}
\begin{proof}[Proof of Theorem \ref{sum formula F3}]
Since \eqref{eq: SF of F2} is obtained from \eqref{eq: SF of F3} and \eqref{eq: SF of F3 star} by taking modulo $x^2_{\cF_3}$, it is sufficient to prove \eqref{eq: SF of F3} and \eqref{eq: SF of F3 star}.
Note that $(-1)^{k}\frZ_{\cF_2}(k+1)x_{\cF_2}=\frZ_{\cF_2}(k+1)x_{\cF_2}$ holds because if $k$ is odd, then $\frZ_{\cF_2}(k+1)x_{\cF_2}=0$.

Let us prove \eqref{eq: SF of F3}. By Theorem~\ref{thm:dep-1_general}, we have
\begin{align}\label{single F3}
\zeta^{}_{\cF_3}(k)=(-1)^k\left\{k\frZ_{\cF_3}(k+1)x_{\cF_3}+\binom{k+1}{2}\frZ_{\cF_3}(k+2)x_{\cF_3}^2\right\}.
\end{align}
Since $x_{\cF_3}^l$ with $l\geq 3$ vanishes, by \eqref{sym sum non star}, we have
\begin{align*}
S_{\cF_3;k,r}&=\frac{1}{r!}\sum_{\bk\in I_{k,r}}\sum_{\sigma\in\frS_r}\zeta_{\cF_3}(\sigma(\bk))\\
&=\frac{1}{r!}\sum_{\bk\in I_{k,r}}\Biggl\{(-1)^{r-1}(r-1)!\zeta^{}_{\cF_3}(k)\\
&\qquad\qquad+(-1)^{r-2}\sum_{\substack{B_1 \sqcup B_2=\{1, \ldots, r\} \\ B_1, B_2 \neq \varnothing}}(\#B_1-1)!(\#B_2-1)!\zeta^{}_{\cF_3}(b_1(\bk))\zeta^{}_{\cF_3}(b_2(\bk))\Biggr\}.
\end{align*}
We calculate the right-hand side. Since $\# I_{k,r}=\binom{k-1}{r-1}$, by \eqref{single F3}, we have
\[
\sum_{\bk\in I_{k,r}}\frac{(-1)^{r-1}}{r}\zeta^{}_{\cF_3}(k)=(-1)^{k+r-1}\left\{\binom{k}{r}\frZ_{\cF_3}(k+1)x_{\cF_3}+\frac{k+1}{2}\binom{k}{r}\frZ_{\cF_3}(k+2)x_{\cF_3}^2\right\}.
\]
Furthermore, since $\#\{\bk=(k_1,\dots, k_r)\in I_{k,r} \mid \sum_{i\in B_1}k_i=b_1\}=\#I_{b_1, \#B_1}\cdot \#I_{b_2, \#B_2}$ for $B_1, B_2\neq\varnothing$ with $B_1\sqcup B_2=\{1,\dots, r\}$ and $b_1, b_2$ with $b_1+b_2=k$, $b_1\geq \#B_1$, $b_2\geq \#B_2$, we have 
\begin{align*}
&\sum_{\bk\in I_{k,r}}\sum_{\substack{B_1 \sqcup B_2=\{1, \ldots, r\} \\ B_1, B_2 \neq \varnothing}}(\#B_1-1)!(\#B_2-1)!\zeta^{}_{\cF_3}(b_1(\bk))\zeta^{}_{\cF_3}(b_2(\bk))\\
&=\sum_{\substack{B_1 \sqcup B_2=\{1, \ldots, r\} \\ B_1, B_2 \neq \varnothing}}\sum_{\substack{b_1+b_2=k \\ b_1\geq \#B_1, b_2\geq \#B_2}}\sum_{\substack{\bk=(k_1,\dots,k_r) \in I_{k,r} \\ \sum_{i\in B_1}k_i=b_1}}(\#B_1-1)!(\#B_2-1)!\zeta^{}_{\cF_3}(b_1)\zeta^{}_{\cF_3}(b_2)\\
&=(-1)^k\sum_{\substack{B_1 \sqcup B_2=\{1, \ldots, r\} \\ B_1, B_2 \neq \varnothing}}\sum_{\substack{b_1+b_2=k \\ b_1\geq \#B_1, b_2\geq \#B_2}}(b_1)_{\#B_1}(b_2)_{\#B_2}\cdot\frZ_{\cF_3}(b_1+1)\frZ_{\cF_3}(b_2+1)x_{\cF_3}^2.
\end{align*}
Note that all terms of $x_{\cF_3}^l$ with $l\geq3$ for $\cF_3$-MZVs vanish. 
This completes the calculation for \eqref{eq: SF of F3}.
The formula \eqref{eq: SF of F3 star} is obtained by a similar calculation using \eqref{sym sum star}.
\end{proof}
%
\subsection{$\cF_2$-sum formula over $I_{k,r,i}$}
In this subsection, we prove the $\cF_2$-sum formula over $I_{k,r,i}$.
For positive integers $k, r, i$ with $1\leq i \leq r<k$, let $I_{k, r, i}$ denote the set of indices $\bk=(k_1,\dots,k_r)$ with $\wt(\bk)=k$, $\dep(\bk)=r$ and $k_i\geq 2$. For $\bullet \in \{\varnothing, \star\}$ and a positive integer $n$, set
\begin{align*}
S^{\bullet}_{\cF_n; k, r, i}\coloneqq \sum_{\bk \in I_{k, r, i}}\zeta^{\bullet}_{\cF_n}(\bk).
\end{align*}

Saito and Wakabayashi \cite{SW1} ($\cF=\cA$) and Murahara \cite{Mur} ($\cF=\cS$) proved that 
\begin{align*}
S_{\cF_1; k, r, i}&=(-1)^{i}\left\{\binom{k-1}{i-1}+(-1)^r\binom{k-1}{r-i}\right\}\frZ_{\cF_1}(k), \\
S^\star_{\cF_1; k, r, i}&=(-1)^{i}\left\{\binom{k-1}{r-i}+(-1)^r\binom{k-1}{i-1}\right\}\frZ_{\cF_1}(k).
\end{align*}

If $k$ is even, then we have $S_{\cF_1; k, r, i}=S^\star_{\cF_1; k, r, i}=0$ by $\frZ_{\cF_1}(k)=0$.
Thus it is a natural question what is a lifting of $S^\bullet_{\cF_1; k,r,i}$ to $\cF_2$, that is, $S^\bullet_{\cF_2; k, r, i}$.
We give the answer in the following form.
\begin{theorem}[{$\cF_2$-sum formula for $I_{k,r,i}$}]\label{i adm sum formula}
Let $k, r, i$ be positive integers with $1\leq i \leq r<k$ and suppose that $k$ is even. Then we have
\begin{align*}
S_{\cF_2; k, r, i}=(-1)^{r-1}\frac{b_{k, r, i}}{2}\cdot\frZ_{\cF_2}(k+1)x_{\cF_2},
\quad S^{\star}_{\cF_2; k, r, i}=\frac{b^\star_{k, r, i}}{2}\cdot\frZ_{\cF_2}(k+1)x_{\cF_2},
\end{align*}
where
\begin{align*}
b_{k, r, i}\coloneqq \binom{k-1}{r}+(-1)^{r-i}\left\{(k-r)\binom{k}{i-1}+\binom{k-1}{i-1}+(-1)^{r-1}\binom{k-1}{r-i}\right\}
\end{align*}
and
\begin{align*}
b^\star_{k, r, i}\coloneqq \binom{k-1}{r}+(-1)^{i-1}\left\{(k-r)\binom{k}{r-i}+\binom{k-1}{r-i}+(-1)^{r-1}\binom{k-1}{i-1}\right\}.
\end{align*}
\end{theorem}
The case $\cF=\cA$ of Theorem~\ref{i adm sum formula} was proved by the third author and Yamamoto \cite{SY}.
In this subsection, we reprove their result and prove the case $\cF=\cS$ simultaneously by a different method.
\begin{lemma}[Recurrence relations]
For positive integers $k, r, i$ with $2\leq i+1\leq r\leq k-1$, we have
\begin{align*}
(r-i)S_{\cF_2; k, r, i}+iS_{\cF_2; k, r, i+1}+(k-r)S_{\cF_2; k, r-1, i}&=\sum_{l=1}^{k-r}\zeta^{}_{\cF_2}(l)S_{\cF_2; k-l, r-1, i}, \\
(r-i)S^\star_{\cF_2; k, r, i}+iS^\star_{\cF_2; k, r, i+1}-(k-r)S^\star_{\cF_2; k, r-1, i}&=\sum_{l=1}^{k-r}\zeta^{}_{\cF_2}(l)S^\star_{\cF_2; k-l, r-1, i}.
\end{align*}
\end{lemma}
\begin{proof}
Let $\bullet\in\{\varnothing, \star\}$.
From the same argument in \cite[Lemma 2.1, Proposition 2.2]{SW1}, we see that the sum of the product
\begin{align*}
\sum_{(k_1, \ldots, k_{r-1}, l) \in I_{k, r, i}}\zeta^{}_{\cF_2}(l)\zeta^\bullet_{\cF_2}(k_1, \ldots, k_{r-1})=\sum_{l=1}^{k-r}\zeta^{}_{\cF_2}(l)S^\bullet_{\cF_2; k-l, r-1, i}
\end{align*}
coincides with the left-hand side of the desired recurrence relation by the harmonic relation for $\cF_2$-MZVs.
\end{proof}
\begin{corollary}\label{cor:reccurence}
If $k$ is even, then we have
\begin{align*}
(r-i)S_{\cF_2; k, r, i}+iS_{\cF_2; k, r, i+1}+(k-r)S_{\cF_2; k, r-1, i}&=0, \\
(r-i)S^\star_{\cF_2; k, r, i}+iS^\star_{\cF_2; k, r, i+1}-(k-r)S^\star_{\cF_2; k, r-1, i}&=0
\end{align*}
for positive integers $k, r, i$ with $2\leq i+1\leq r\leq k-1$.
\end{corollary}
\begin{proof}
If $l$ is odd, then $\zeta^{}_{\cF_2}(l)=0$ by Theorem~\ref{thm:dep-1_general} or \eqref{rep.k non star}.
If $l$ is even, then $\zeta^{}_{\cF_2}(l)$ is a multiple of $x_{\cF}$ by Theorem~\ref{thm:dep-1_general} or \eqref{rep.k non star} and $S^{\bullet}_{\cF_2; k-l, r-1, i}$ is also a multiple of $x_{\cF}$ by Saito--Wakabayashi and Murahara's sum formulas.
\end{proof}
\begin{proof}[Proof of Theorem \ref{i adm sum formula}]
We prove the non-star case by backward induction on $r\le k-1$.
Since
\begin{align*}
b_{k, k-1, i}
&=\binom{k-1}{k-1}+(-1)^{k-1-i}\left\{(k-k+1)\binom{k}{i-1}+\binom{k-1}{i-1}+(-1)^{k-2}\binom{k-1}{k-1-i}\right\}\\
&=1+(-1)^{i-1}\binom{k+1}{i},
\end{align*}
we have 
\begin{align*}
S_{\cF_2; k, k-1, i}&=\zeta^{}_{\cF_2}(\{1\}^{i-1}, 2, \{1\}^{k-i-1})=\frac{1}{2}\left\{1+(-1)^{i-1}\binom{k+1}{i}\right\}\frZ_{\cF_2}(k+1)x_{\cF_2}\\
&=\frac{b_{k, k-1, i}}{2}\cdot\frZ_{\cF_2}(k+1)x_{\cF_2},
\end{align*}
by the definition of $S_{\cF_2; k, r, i}$ and \eqref{S2 1 2 1 non star}.
Hence, the case $r=k-1$ is true.
To complete the induction step, by Corollary~\ref{cor:reccurence}, it suffices to prove that 
\begin{equation}\label{ind step}
(r-i)b_{k, r, i}+ib_{k, r, i+1}-(k-r)b_{k, r-1, i}=0
\end{equation}
holds for $2 \leq r\leq k-1$.
The left-hand side of \eqref{ind step} is
\begin{align*}
(r-i)&\left[\binom{k-1}{r}+(-1)^{r-i}\left\{(k-r)\binom{k}{i-1}+\binom{k-1}{i-1}+(-1)^{r-1}\binom{k-1}{r-i}\right\}\right] \\
+i&\left[\binom{k-1}{r}+(-1)^{r-i-1}\left\{(k-r)\binom{k}{i}+\binom{k-1}{i}+(-1)^{r-1}\binom{k-1}{r-i-1}\right\}\right] \\
-(k-r)&\left[\binom{k-1}{r-1}+(-1)^{r-i-1}\left\{(k-r+1)\binom{k}{i-1}+\binom{k-1}{i-1}+(-1)^{r-2}\binom{k-1}{r-i-1}\right\}\right]
\end{align*}
by definition.
By $\binom{k-1}{r-1}=\frac{r}{k-r}\binom{k-1}{r}$, we have
\begin{equation}\label{first terms}
(r-i)\binom{k-1}{r}+i\binom{k-1}{r}-(k-r)\binom{k-1}{r-1}=0.
\end{equation}
By $\binom{k}{i}=\frac{k-i+1}{i}\binom{k}{i-1}$, we have
\begin{equation}\label{second terms}
(r-i)(k-r)\binom{k}{i-1}-i(k-r)\binom{k}{i}+(k-r)(k-r+1)\binom{k}{i-1}=0.
\end{equation}
By $\binom{k-1}{i}=\frac{k-i}{i}\binom{k-1}{i-1}$, we have
\begin{equation}\label{third terms}
(r-i)\binom{k-1}{i-1}-i\binom{k-1}{i}+(k-r)\binom{k-1}{i-1}=0.
\end{equation}
By $\binom{k-1}{r-i}=\frac{k-r+i}{r-i}\binom{k-1}{r-i-1}$, we have
\begin{equation}\label{fourth terms}
(r-i)\binom{k-1}{r-i}-i\binom{k-1}{r-i-1}-(k-r)\binom{k-1}{r-i-1}=0.
\end{equation}
From \eqref{first terms}, \eqref{second terms}, \eqref{third terms} and \eqref{fourth terms}, we obtain \eqref{ind step} and we complete the proof of the formula for $S_{\cF_2; k, r, i}$. 
In the star case, we should prove
\begin{align*}
S^{\star}_{\cF_2; k, k-1, i}=\frac{b^{\star}_{k, k-1, i}}{2}\cdot \frZ_{\cF_2}(k+1)x_{\cF_2}
\end{align*}
and the recurrence relation
\begin{align*}
(r-i)b^\star_{k, r, i}+ib^\star_{k, r, i+1}-(k-r)b^\star_{k, r-1, i}=0.
\end{align*}
These are proved similarly to the non-star case.
\end{proof}
\begin{remark}
We can also prove the star case by connecting $S_{\cF_2; k, r, i}$ and $S^{\star}_{\cF_2; k, r, i}$ directly using Proposition \ref{antipode}. This is the method used in \cite{SY}.
\end{remark}
%
\appendix
\section{Proof of equality~\eqref{key binom calc}}
In this appendix, we prove the following proposition.
\begin{proposition}\label{prop:true value of C}
For non-negative integers $a$ and $b$, we have
\begin{align*}
C=1+(-1)^a\binom{a+b+3}{b+2};
\end{align*}
see the proof of Theorem~\ref{S2 1 2 1} for the definition of $C$.
\end{proposition}
We divide $C$ into six parts. Set
\begin{align*}
\text{I}&\coloneqq(-1)^{a+1}\sum_{\substack{l+m=b-1 \\ l, m\geq0}}\sum_{\substack{r+s=a \\ r, s\geq0}}(-1)^{s+m}\binom{r+l+1}{r}\binom{s+m+1}{s}\binom{a+b+3}{s+m+2},\\
\text{II}&\coloneqq(-1)^{a+1}\sum_{\substack{l+m=b-1 \\ l, m\geq0}}\sum_{\substack{r+s=a \\ r, s\geq0}}\binom{r+l+1}{r}\binom{s+m+1}{s}, \\
\text{III}&\coloneqq2(-1)^a\sum_{\substack{r+s=a \\ r, s\geq0}}(-1)^s\binom{r+b+1}{r}\binom{a+b+3}{s+1},\\
\text{IV}&\coloneqq2(-1)^{a+1}\sum_{\substack{r+s=a \\ r, s\geq0}}\binom{r+b+1}{r},\\
\text{V}&\coloneqq\sum_{\substack{m+n=a-1 \\ m, n\geq0}}(-1)^n\sum_{\substack{r+s=n \\ r, s\geq0}}(-1)^{s+m+1}\binom{r+b+1}{r}\binom{s+m+1}{s}\binom{a+b+3}{s+m+2},\\
\text{VI}&\coloneqq\sum_{\substack{m+n=a-1 \\ m, n\geq0}}(-1)^{n+1}\sum_{\substack{r+s=n \\ r, s\geq0}}\binom{r+b+1}{r}\binom{s+m+1}{s}.
\end{align*}
Note that we easily obtain
\begin{align}\label{eq: IV}
\text{IV}=2(-1)^{a+1}\binom{a+b+2}{a}.
\end{align}
By the definition of negative binomial coefficients and the Chu--Vandermonde identity, we also have
\begin{align}\label{eq: II and VI}
\text{II}=(-1)^{a+1}b\binom{a+b+2}{a}, 
\qquad \text{VI}=(-1)^a\binom{a+b+1}{a-1}.
\end{align}

Next, we calculate I, III and V.  We use the following equality repeatedly: 
\begin{align}\label{eq: PFD}
  \sum_{k=0}^n(-1)^k\binom{n}{k}\frac{1}{x+k}=\frac{n!}{x(x+1)\cdots(x+n)}.
\end{align}
Here, $n$ is a non-negative integer and $x$ is an indeterminate.
\begin{lemma}\label{I}
    \begin{align*}
    \text{I}=0.
    \end{align*}
\end{lemma}
\begin{proof}
By elementary calculations, we have
\begin{equation}\label{first calc of I}
\begin{split}
  \text{I}&=(-1)^{a+1}\sum_{m=0}^{b-1}\sum_{s=0}^a(-1)^{s+m}\binom{a+b-m-s}{b-m}\binom{s+m+1}{s}\binom{a+b+3}{s+m+2}\\           
  &=(-1)^{a+1}(a+1)\binom{a+b+2}{a+1}\sum_{s=0}^a(-1)^s\binom{a}{s}  \\
                & \quad \times\sum_{m=0}^{b-1}(-1)^m\binom{b+1}{m+1}\left(\frac{1}{a+b+1-s-m}+\frac{1}{s+m+2}\right).
\end{split}
\end{equation}
Applying \eqref{eq: PFD} with $x=-(a+b+2-s)$ and $x=s+1$, we have
\begin{equation}\label{second calc of I}
\begin{split}
  &\sum_{m=0}^{b-1}(-1)^m\binom{b+1}{m+1}\left(\frac{1}{a+b+1-s-m}+\frac{1}{s+m+2}\right)\\
  &=\sum_{m=0}^{b+1}(-1)^{m-1}\binom{b+1}{m}\left(\frac{1}{a+b+2-s-m}+\frac{1}{s+m+1}\right) \\
  &\quad+\left(\frac{1}{a+b+2-s}+\frac{1}{s+1}\right)+(-1)^{b+1}\left(\frac{1}{a+1-s}+\frac{1}{s+b+2}\right) \\
  &=\frac{(-1)^b}{a+b+2-s}\binom{a+b+1-s}{b+1}^{-1}-\frac{1}{s+1}\binom{b+s+2}{b+1}^{-1}\\
  &\quad+\frac{1}{a+b+2-s}+\frac{1}{s+1} +\frac{(-1)^{b+1}}{a+1-s}+\frac{(-1)^{b+1}}{s+b+2}. 
\end{split}
\end{equation}
By substituting \eqref{second calc of I} into \eqref{first calc of I} and then using \eqref{eq: PFD} again, we have
\begin{align*}
&(-1)^{a+b+1}\sum_{s=0}^a(-1)^s\binom{a+b+2}{s}=(-1)^{b+1}\binom{a+b+1}{a},\\
&(-1)^a\sum_{s=0}^a(-1)^s\binom{a+b+2}{a-s}=\sum_{s=0}^a(-1)^s\binom{a+b+2}{s}=(-1)^a\binom{a+b+1}{a},\\
&(-1)^{a+1}(a+1)\binom{a+b+2}{a+1}\sum_{s=0}^a(-1)^s\binom{a}{s}\frac{1}{a+b+2-s}=-1,\\
&(-1)^{a+1}(a+1)\binom{a+b+2}{a+1}\sum_{s=0}^a(-1)^s\binom{a}{s}\frac{1}{s+1}=(-1)^{a+1}\binom{a+b+2}{a+1},\\
&(-1)^{a+b}(a+1)\binom{a+b+2}{a+1}\sum_{s=0}^a(-1)^s\binom{a}{s}\frac{1}{a+1-s}=(-1)^b\binom{a+b+2}{a+1},\\
&(-1)^{a+b}(a+1)\binom{a+b+2}{a+1}\sum_{s=0}^a(-1)^s\binom{a}{s}\frac{1}{s+b+2}=(-1)^{a+b}.
\end{align*}
Since $a+b$ is even, we have the conclusion.
\end{proof}
\begin{lemma}\label{lem:III}
\[
  \text{III}=2+2(-1)^a\binom{a+b+2}{a+1}.
\]
\end{lemma}
\begin{proof}
By \eqref{eq: PFD}, we have
\begin{align*}
  \text{III}
     &=2(-1)^a\sum_{s=0}^a(-1)^s\binom{a+b+1-s}{a-s}\binom{a+b+3}{s+1}\\
     &=2(-1)^{a-1}(a+b+3)\binom{a+b+2}{a+1}\left\{\sum_{s=0}^{a+1} (-1)^s\binom{a+1}{s}\frac{1}{a+b+3-s}-\frac{1}{a+b+3}\right\}\\
     &=2(-1)^{a-1}(a+b+3)\binom{a+b+2}{a+1}\left\{\frac{(-1)^{a-1}}{a+b+3}\binom{a+b+2}{a+1}^{-1}-\frac{1}{a+b+3}\right\}\\
     &=2+2(-1)^a\binom{a+b+2}{a+1},
  \end{align*}
which completes the proof.
\end{proof}
\begin{lemma}\label{V}
  \begin{align*}
  \text{V}=(-1)^aa\binom{a+b+2}{a+1}+(-1)^a\binom{a+b+1}{a}-1.
  \end{align*}
\end{lemma}
\begin{proof}
Since
  \begin{align*}
  \text{V}
    &=(-1)^a\sum_{n=0}^{a-1}\sum_{s=0}^n(-1)^s\binom{b+n+1-s}{b+1}\binom{a+s-n}{s}\binom{a+b+3}{a+s+1-n}\\
    &=(-1)^a(a+1)\binom{a+b+2}{a+1}\sum_{n=0}^{a-1}\binom{a}{n}\\
    & \quad \times \sum_{s=0}^n(-1)^s\binom{n}{s}\left(\frac{1}{a+s+1-n}+\frac{1}{b+2+n-s}\right)
  \end{align*}
and
  \begin{align*}
  \sum_{s=0}^n(-1)^s\binom{n}{s}\frac{1}{a+s+1-n}&=\frac{1}{a+1-n}\binom{a+1}{n}^{-1},\\
  \sum_{s=0}^n(-1)^s\binom{n}{s}\frac{1}{b+2+n-s}&=\frac{(-1)^n}{b+2+n}\binom{b+n+1}{b+1}^{-1}
  \end{align*}
hold by \eqref{eq: PFD}, we obtain the desired formula.
\end{proof}
\begin{proof}[Proof of Proposition \ref{key binom calc}]
From \eqref{eq: IV}, \eqref{eq: II and VI},  Lemmas \ref{I}, \ref{lem:III}, and \ref{V}, we obtain the desired formula.
\end{proof}
%

\end{document}